\documentclass[11pt]{article}

\usepackage[margin=1in]{geometry}
\usepackage{amsmath,amssymb,amsthm,mathtools}
\usepackage{enumitem}
\usepackage{xcolor}
\usepackage{microtype}
\usepackage{hyperref}

\hypersetup{
  colorlinks=true,
  linkcolor=blue!55!black,
  citecolor=blue!55!black,
  urlcolor=blue!55!black
}

\newtheorem{theorem}{Theorem}[section]
\newtheorem{proposition}[theorem]{Proposition}
\newtheorem{lemma}[theorem]{Lemma}
\newtheorem{corollary}[theorem]{Corollary}
\newtheorem{remark}[theorem]{Remark}
\newtheorem{definition}[theorem]{Definition}
\newtheorem{assumption}[theorem]{Assumption}
\newtheorem{example}[theorem]{Example}
\newtheorem{algorithm}[theorem]{Algorithm}

\DeclareMathOperator{\gph}{gph}
\DeclareMathOperator{\dom}{dom}
\DeclareMathOperator{\rge}{rge}
\DeclareMathOperator{\ri}{ri}
\DeclareMathOperator{\zer}{zer}
\DeclareMathOperator{\dist}{dist}
\DeclareMathOperator{\argmin}{argmin}
\DeclareMathOperator{\cl}{cl}

\DeclareMathOperator{\rec}{rec}

\newcommand{\HH}{\mathcal H}
\newcommand{\RR}{\mathbb R}

\newcommand{\eps}{\varepsilon}
\newcommand{\inner}[2]{\langle #1,#2\rangle}
\newcommand{\norm}[1]{\|#1\|}
\newcommand{\set}[1]{\{#1\}}
\newcommand{\To}{\rightrightarrows}

\title{Semidefinite Implementations of the Proximal Point Algorithm\\
       and Consequences for Progressive Decoupling in Optimization}

\author{
Xudong Li\thanks{School of Data Science, Fudan University, Shanghai, China. Email: \texttt{lixudong@fudan.edu.cn}.}
\and
R. Tyrrell Rockafellar\thanks{Department of Mathematics, University of
Washington, Seattle, WA 98195-4350, USA. Email: \texttt{rtr@uw.edu}.}
\and
Defeng Sun\thanks{Department of Applied Mathematics, The Hong Kong Polytechnic
University, Hung Hom, Hong Kong. Defeng Sun was supported by the Research Center for Intelligent Operations Research and  RGC Senior Research Fellow Scheme No. SRFS2223-5S02.
Email: \texttt{defeng.sun@polyu.edu.hk}.}
}
\date{Draft of September 20, 2026}

\begin{document}
\maketitle

\begin{abstract}
The proximal point algorithm for finding a zero of a maximal monotone
mapping is the engine that drives augmented Lagrangian methods and various
splitting schemes in convex optimization, where proximal terms that are 
merely positive semidefinite can be helpful in preconditioning.  Here 
a semidefinite extension of that algorithm is developed with new features 
which guarantee convergence, typically even linear convergence, in broader 
territory than previously foreseen.  Moreover, inexact minimization is 
allowed in the subproblems in a variable-metric implementation which might 
open up quasi-Newton-like approaches. 

The key to the extension is applying an advanced form of the positive
definite proximal point algorithm to a projected monotone mapping 
derived from the given mapping, and combining that with a completion step.  
For this to work, the maximality of the projected mapping is essential, but 
maximality is not automatic.  The circumstances that guarantee maximality 
are definitively analyzed and convenient criteria for maximality are 
thereby identified.  Consequences of the extended algorithm are worked out 
for application to convex inf-projection, augmented Lagrangian steps, and 
bracketed progressive decoupling based on Spingarn's partial inverse as a 
scheme for problem decomposition.

\end{abstract}


\section{Introduction}

Let $T$ be maximal monotone on a finite-dimensional Hilbert
space $\HH$.  The proximal point algorithm (PPA) seeks a zero of $T$ through
solving inclusions of the form
\begin{equation}\label{eq:intro-semidefinite-ppa}
  0\in T(z^{k+1})+c_k^{-1}B_k(z^{k+1}-z^k),
\end{equation}
where $c_k>0$ and $B_k$ is self-adjoint.  When $B_k$ is positive definite,
the resolvent, namely the mapping from $z^k$ to $z^{k+1}$, is everywhere 
defined, and convergence under exact or summably inexact evaluation in the
face of metric variation is well understood
\cite{Rockafellar1976PPA,LiST2020LinearProgramming,Rockafellar2023VariableMetric}.
Positive semidefinite choices are nevertheless natural when only selected
primal or dual directions can be regularized without destroying a useful
decomposition.  The automatic well-posedness of the resolvent mapping in 
the positive definite setting is then lost:  maximal monotonicity of $T$, 
even together with $\zer T\ne\emptyset$, does not ensure 
\emph{implementability}, meaning the solvability of 
\eqref{eq:intro-semidefinite-ppa} for $z^{k+1}$ from every point $z^k\in\HH$. 
Nor does implementability entail uniqueness of $z^{k+1}$, since the 
regularization provided by $B_k$ is oblivious to directions tied to its 
kernel.  Two distinct questions must therefore be answered in contemplating 
any semidefinite extension of the proximal point algorithm.  What guarantees
its implementability and how can the aspects of the iterates in unregularized
directions be controlled?

In practice, the self-adjoint positive semidefinite operators $B_k$ can be 
supposed to have the same kernel and therefore the same range, those two
subspaces being orthogonally complementary to each other.  In denoting the 
common range by $H_*$ and the common kernel by $H_{**}$, we say that
the operators in this case are all \emph{aligned} with a specific
``splitting'' of $\HH$ into the direct sum $H_*\oplus H_{**}$ of two
smaller Hilbert spaces.  We show that each semidefinite proximal point 
step turns out then to reduce to a positive definite proximal point 
step for a ``projected'' mapping $T_*$ from $H_*$ into $H_*$ followed by 
a selection from a ``completion'' mapping $T_{**}$ from $H_*\times H_*$ into $H_{**}$.  
Our first main result gives an exact characterization of implementability 
in terms of this reduction:   implementability is equivalent to maximal 
monotonicity of $T_*$ as derived from $T$ and the splitting of $\HH$ into 
$H_*\oplus H_{**}$; it is independent of the particular choices of $B_k$ 
and $c_k$.  The component of $z^{k+1}$ in $H_*$, the ``active'' update, is 
uniquely determined, although the component of $z^{k+1}$ in $H_{**}$, the 
``completion'' update, might not be.  This reduction scheme is central to 
our analysis, but in algorithmic implementation can remain behind the 
scenes, hidden from view. 

Although maximal monotonicity of the projected mapping $T_*$ settles 
existence and uniqueness of the active update in $H_*$ and the 
nonemptiness of the image sets of $T_{**}$ in $H_{**}$ from which the 
completion update is a selection, it says nothing about potential uniqueness or
even boundedness of those sets, which we call ``completion fibers.''  Our 
second main result illuminates the global geometry of the completion
fibers:  they are closed convex sets that share a common recession cone.  
Consequently, one completion fiber is bounded if and only if every 
completion fiber is bounded.  The \emph{one-bounded-fiber} condition 
guarantees moreover, not only that every fiber is compact, but also that 
$T_{**}$ is locally bounded and $T_*$ is maximal monotone.  It can 
conveniently serve that way also as a criterion for implementability.
It provides control of the directions in which the proximal term fails
to regularize, without imposing a single-valuedness restriction on the 
semidefinite resolvent mapping in \eqref{eq:intro-semidefinite-ppa}. 

Several related studies formulate well-posedness of the resolvent directly 
through the degenerate resolvent $(c_kT+B_k)^{-1}B_k$.  Ha \cite{Ha1990} 
obtains full-domain and convergence results for a one-block generalized 
proximal map under a range-interior condition.  Bredies et al.\ 
\cite{BrediesChencheneLorenzNaldi2022} work with an admissible fixed
preconditioner, for which the full resolvent is everywhere defined and
single-valued.  Variable-metric and inexact degenerate schemes have also 
been studied under additional full-space regularity assumptions 
\cite{LorenzMarquardtNaldi2025,MarquesAlvesLorenzNaldi2026}.  More recently, 
\cite{XueZhang2026Restricted} characterizes full-domain solvability of a 
fixed degenerate resolvent through maximal monotonicity of a restricted 
inverse and analyzes conditions for its single-valuedness.  Our starting 
point is different.  Rather than treating the full degenerate resolvent as 
the primary object, we decompose each semidefinite proximal step from the 
outset into an active proximal update and a compatible kernel completion,
as already indicated. This makes implementability a property solely of the 
active dynamics, while boundedness and uniqueness of full-space completions 
become separate questions.

The characterization also admits useful verification routes. A classical relative-interior condition and a piecewise-polyhedral projection result \cite{LiRockafellarSun2026Polyhedral} provide sufficient conditions for maximal monotonicity of $T_*$. In particular, if $T$ is maximal monotone and piecewise polyhedral with a nonempty zero set, then implementability holds for every fixed splitting. Once implementability is established, the active iteration is simply a positive definite variable-metric PPA for $T_*$. Under the standing zero-existence and metric assumptions, summable active residual errors yield convergence of the active iterates to a zero of $T_*$. Control of the full-space completions is provided separately by the one-bounded-fiber condition: it makes every completion fiber compact and locally stabilizes the completion mapping, and consequently yields boundedness of the full iterates. Both Ha's range-interior condition \cite{Ha1990} and full-space admissibility \cite{BrediesChencheneLorenzNaldi2022} imply the one-bounded-fiber condition, but the converse fails, i.e., bounded-fiber control still permits nonunique completions. Finally, under a stronger relative residual rule, metric subregularity of $T_*$ yields Q-linear convergence of the algorithm.

Two convex specializations make the reduction explicit.  For $T=\partial f$ with $f$ being proper closed convex,
a nonempty bounded partial-minimizer set at one active point makes the
inf-projection of $f$ closed, proper, and convex. Its subdifferential is exactly $T_*$,
and the completion fibers are the corresponding partial-minimizer sets.  For
the convex Lagrangian saddle mapping, an augmented Lagrangian step is identified as a
semidefinite proximal step.  When the dual objective is proper,
implementability is equivalent to exact primal realizability of every dual
proximal step.

Finally, we apply the framework to bracketed progressive decoupling through
Spingarn's partial inverse \cite{Spingarn1983,Rockafellar2019Progressive,Rockafellar2019ProgressiveD,
Rockafellar2024PMM}.  The brackets retain proximal regularization only on a
prescribed active subspace, producing a  semidefinite PPA for the
partial inverse and an equivalent realization in the original variables.
Its implementability is characterized by maximal monotonicity of the
projected partial inverse.  For linearly constrained block-separable
piecewise linear--quadratic models admitting a KKT pair, this condition holds
for every pair of brackets.  

The remainder of the paper is organized as follows.
Section~\ref{sec:projection} introduces implementability, derives the
active--completion reduction, and gives its projected-maximality
characterization.
Section~\ref{sec:conditions} gives conditions for projected maximality,
studies completion stability, and compares the one-bounded-fiber condition
with Ha's range condition and full-space admissibility.
Section~\ref{sec:convergence} develops the active and full-space convergence
conclusions and then treats convex inf-projection and exact augmented
Lagrangian steps.
Section~\ref{sec:applications} applies the framework to bracketed progressive
decoupling.
Section~\ref{sec:conclusion} concludes the paper.

\medskip

\noindent\textbf{Notation.}
All spaces are finite-dimensional Hilbert spaces.  The ambient space $\HH$
has inner product $\inner{\cdot}{\cdot}$ and norm $\norm{\cdot}$.  For an orthogonal decomposition $\HH=H_*\oplus H_{**}$, we use the
canonical isometric identification
\[
  z=z_*+z_{**}\longleftrightarrow(z_*,z_{**})\in H_*\times H_{**},
\]
and regard $H_*$ and $H_{**}$ as $H_*\times\{0\}$ and
$\{0\}\times H_{**}$, respectively.  
We use
the standard notation $\gph T$, $\dom T$, $\rge T$, $T^{-1}$, and
$\zer T:=\set{z\mid0\in T(z)}$ for a set-valued mapping $T$, and $\ri C$,
$N_C$, and $\dist(\cdot,C)$ for a closed convex set $C$. The recession cone of a nonempty closed convex set $C\subset \HH$ is
$\rec C:=\set{d\in \HH \mid C+td\subset C\ \forall t\ge0}$.

Let $X$ and $Y$ be finite-dimensional Hilbert spaces. 
For a self-adjoint positive semidefinite linear operator $G\succeq0$ on $X$, we write
\(
  \norm{u}_G:=\sqrt{\inner{Gu}{u}}
\)
for the induced seminorm. When $G\succ0$, this is a norm, and we set
\(
  \dist_G(x,C):=\inf_{v\in C}\norm{x-v}_G,
\)
where an infimum over the empty set is $+\infty$.
For $G\succ0$ on $X$, a mapping $F:X\To X$ is \emph{metrically
subregular at $\bar x$ for $\bar y$ in the $G$-metric} if there exist a
neighborhood $N$ of $(\bar x,\bar y) \in \gph F$ and $a\in(0,+\infty)$ such that
\[
  \dist_G\bigl(x,F^{-1}(\bar y)\bigr)
  \le a\norm{y-\bar y}_{G^{-1}}
  \qquad\text{for every }(x,y)\in N\cap\gph F.
\]
The infimum of all such local constants is denoted by
$\operatorname{subreg}_G(F;\bar x\mid\bar y)$, with value $+\infty$ when
metric subregularity fails; see
\cite[(2.16)--(2.17)]{Rockafellar2023VariableMetric}.
A set-valued mapping
$F:X\To Y$ is \emph{locally bounded at} $\bar x\in\dom F$ if
$\bigcup_{x\in U\cap\dom F}F(x)$ is bounded for some neighborhood $U$ of
$\bar x$; it is locally bounded if this holds throughout $\dom F$.

For a proper closed
convex function $h$, its horizon function is
\[
  h^\infty(d):=\lim_{t\to\infty}
  \frac{h(\bar x+td)-h(\bar x)}{t},
  \qquad \bar x\in\dom h,
\]
independently of $\bar x$ \cite[Theorem~3.21]{RockafellarWets1998}.
A function $\varphi:X\times Y\to(-\infty,+\infty]$ is
\emph{level-bounded in $x$ locally uniformly in $y$} \cite[Definition~1.16]{RockafellarWets1998} if, for every
$\bar y\in Y$ and $\alpha\in\mathbb R$, there is a neighborhood $V$ of
$\bar y$ for which
\[
  \bigcup_{y\in V}\set{x\in X\mid\varphi(x,y)\le\alpha}
  \quad\text{is bounded}.
\]

\section{Implementability and projected reduction}\label{sec:projection}

Fix a monotone mapping $T:\HH\To\HH$.  Throughout, we assume that the self-adjoint
positive semidefinite operators $B_k$ share a common range and kernel:
\[
  H_*:=\rge B_k,
  \qquad
  H_{**}:=\ker B_k,
  \qquad
  \HH=H_*\oplus H_{**}.
\]
Let $P_*$ and $P_{**}$ be the corresponding
orthogonal projections.  For $z\in\HH$, write
$z_*=P_*z$ and $z_{**}=P_{**}z$, so that
$z=z_*+z_{**}$ and, under the convention above, $z=(z_*,z_{**})$.
We call a self-adjoint positive semidefinite operator $B$ \emph{aligned} with
this splitting if
\[
  \rge B=H_*,
  \qquad
  \ker B=H_{**}.
\]
For such a $B$, 
we write the restriction
$B_{*}:=B|_{H_*}$, which is self-adjoint and positive definite with
\(  B(z_*,z_{**})=(B_{*}z_*,0). \)

For a fixed stepsize and aligned proximal operator, the basic question is
whether the semidefinite proximal inclusion can be solved from every current
point.

\begin{definition}[Implementability]\label{def:implementability}
Fix $c>0$ and an aligned operator $B$.  The corresponding semidefinite 
proximal step for $T$, abbreviated as the $(c,B)$-step, is called 
implementable if, for every $x\in\HH$, the inclusion
\[
  0\in T(z)+c^{-1}B(z-x)
\]
has at least one solution $z\in\HH$.  The semidefinite proximal point 
algorithm with aligned operators $B_k$ is called implementable if the 
$(c_k,B_k)$-step \eqref{eq:intro-semidefinite-ppa} is implementable for 
every $c_k>0$ and choice of $B_k$.
\end{definition}

To characterize this property, 
we start from the following block form of an
aligned $B$: for $z,x\in\HH$,
\begin{equation}\label{eq:full-step-graph-form}
  0\in T(z)+c^{-1}B(z-x)
  \quad\Longleftrightarrow\quad
  \left(-c^{-1}B_*(z_*-x_*),0\right)
  \in T(z_*,z_{**}).
\end{equation}
Relation~\eqref{eq:full-step-graph-form} leads to the projected mapping
$T_*:H_*\To H_*$ and the completion mapping
$T_{**}:H_*\times H_*\To H_{**}$ defined by
\begin{equation}\label{eq:Tstar-def}
  w_*\in T_*(z_*)
  \quad\Longleftrightarrow\quad
  \exists z_{**}\in H_{**}\ \hbox{such that}\ (w_*,0)\in T(z_*,z_{**}),
\end{equation}
and
\begin{equation}\label{eq:Tstarstar-def}
  T_{**}(z_*,w_*):=
  \set{z_{**}\in H_{**}\mid (w_*,0)\in T(z_*,z_{**})}.
\end{equation}
The following theorem makes the active--completion decomposition in
\eqref{eq:full-step-graph-form} precise.

\begin{theorem}\label{thm:step-reduction}
Let $c>0$ and let $B$ be aligned with the splitting $\HH=H_*\oplus H_{**}$.  For
$z,x\in\HH$, write $z=(z_*,z_{**})$ and $x=(x_*,x_{**})$.  Then the
following are equivalent:
\begin{enumerate}[label=(\alph*)]
  \item The full-space inclusion holds:
  \[
    0\in T(z)+c^{-1}B(z-x) \quad \mbox{ or equivalently } \quad z\in (B + cT)^{-1} Bx.
  \]
  \item The active and completion relations hold:
  \begin{align*}
    0 \in T_*(z_*)+c^{-1}B_*(z_*-x_*),\qquad 
    z_{**} \in
    T_{**}\left(z_*,-c^{-1}B_*(z_*-x_*)\right).
  \end{align*}
\end{enumerate}
Moreover, the active relation is equivalent to 
\(
z_* \in (I + c B_*^{-1} T_*)^{-1}x_*.
\)
\end{theorem}

\begin{proof}
  The equivalence of \emph{(a)} and \emph{(b)} is immediate from the definitions of $T_*$ and $T_{**}$ in \eqref{eq:Tstar-def} and \eqref{eq:Tstarstar-def}, and \eqref{eq:full-step-graph-form}. The last assertion follows from direct calculations.
\end{proof}

Definition~\ref{def:implementability} and
Theorem~\ref{thm:step-reduction} therefore give, for given $c>0$ and
aligned $B$,
\[
  \text{the $(c,B)$-step is implementable}
  \quad\Longleftrightarrow\quad
  \dom\left(I+cB_*^{-1}T_*\right)^{-1}=H_*.
\]
Although the active resolvent depends on $c$ and $B$, the property that it
has full domain does not.  Applying Minty's theorem yields the following characterization.

\begin{theorem}
\label{thm:implementability-equivalence}
Let $T:\HH\To\HH$ be monotone. Fix a splitting
$\HH=H_*\oplus H_{**}$, and define $T_*$ by \eqref{eq:Tstar-def}.  Then the
following are equivalent:
\begin{enumerate}[label=(\alph*)]
  \item For some $c>0$ and some aligned $B$, the $(c,B)$-step for $T$ is
  implementable.
  \item The projected mapping $T_*:H_*\To H_*$ is maximal monotone.
  \item For every $c>0$ and every aligned $B$, the $(c,B)$-step for $T$ is
  implementable, i.e., the aligned semidefinite proximal scheme for $T$ is implementable. 
\end{enumerate}
\end{theorem}

\begin{proof}
First note that $T_*$ is monotone.  Indeed, for
$(z_*^i,w_*^i)\in\gph T_*$, choose $z_{**}^i$ such that
$(w_*^i,0)\in T(z_*^i,z_{**}^i)$, $i=1,2$.  Monotonicity of $T$ then gives
\[
  \inner{z_*^1-z_*^2}{w_*^1-w_*^2}
  =
  \inner{(z_*^1,z_{**}^1)-(z_*^2,z_{**}^2)}
  {(w_*^1,0)-(w_*^2,0)}
  \ge0.
\]

Assume first that \emph{(a)} holds, and fix the pair $(c,B)$ supplied there.
By Theorem~\ref{thm:step-reduction}, its implementability is equivalent to
solvability of the active inclusion
\[
  0\in T_*(z_*)+c^{-1}B_*(z_*-x_*)
\]
for every $x_*\in H_*$.  After the invertible change of variables
$\xi=B_*^{1/2}z_*$, set
\[
  y=B_*^{1/2}x_*,
  \qquad
  \widehat T(\xi):=B_*^{-1/2}T_*(B_*^{-1/2}\xi).
\]
The active inclusion becomes the ordinary resolvent inclusion
\[
  0\in \widehat T(\xi)+c^{-1}(\xi-y).
\]
The mapping $\widehat T$ is monotone, and it is maximal monotone if and only
if $T_*$ is maximal monotone.  Minty's theorem says that the resolvent of
$\widehat T$ has full domain precisely in that case \cite{Minty1962}.
Therefore $T_*$ is maximal monotone, and \emph{(a)} implies \emph{(b)}.

Conversely, suppose \emph{(b)} holds and fix any $c>0$ and any aligned
$B$.  The same change of variables makes $\widehat T$ maximal monotone, so
Minty's theorem gives a unique active solution for every $x_*$.  Definition
\eqref{eq:Tstar-def} then supplies at least one compatible completion.
Therefore \emph{(c)} holds.  Finally, \emph{(c)} implies \emph{(a)}, for
example by taking $c=1$ and $B=P_*$.  This completes the proof.
\end{proof}

Theorem~\ref{thm:implementability-equivalence} reduces implementability to 
maximal monotonicity of the projected mapping $T_*$.  This property is not
automatically inherited from the original mapping. Indeed, $T$ may be maximal
monotone, with $\zer T\ne\emptyset$, while $T_*$ is not maximal monotone, as
the next example shows.

\begin{example}
\label{ex:maximal-not-implementable}
The following example is adapted from
\cite[Remark~2.2]{BrediesChencheneLorenzNaldi2022}.  On $\RR^2$, let
\[
  f(x,y):=\max\{e^y-x,0\},
  \qquad
  T:=\partial f,
  \qquad
  M:=\operatorname{diag}(1,0).
\]
Then $T$ is maximal monotone with $\zer T\ne\emptyset$, whereas the
projected mapping $T_*$ induced by $M$ is not maximal monotone.  Thus maximal
monotonicity of $T$, even
together with $\zer T\ne\emptyset$, does not guarantee implementability for
one---and hence for every---$c>0$ and aligned $B$.
\end{example}

\begin{proof}[{\bf Detail}]
The splitting induced by $M$ is
$H_*=\RR\times\{0\}$ and $H_{**}=\{0\}\times\RR$.  For compactness, suppress
the standard axis embeddings $s\mapsto(s,0)\in H_*$ and
$t\mapsto(0,t)\in H_{**}$ and write scalar coordinates for the arguments
and values of $T_*$ and $T_{**}$.
The subdifferential of $f$ is
\[
  \partial f(x,y)=
  \begin{cases}
    \{(0,0)\}, & e^y<x,\\
    \{(-1,e^y)\}, & e^y>x,\\
    \set{(-\lambda,\lambda e^y)\mid0\le\lambda\le1}, & e^y=x.
  \end{cases}
\]
It follows that
\[
  \zer T=\set{(x,y)\mid x>0,\ y\le\log x}\ne\emptyset.
\]
Requiring the second component to vanish gives
\[
  T_*(x)=
  \begin{cases}
    \{0\}, & x>0,\\
    \emptyset, & x\le0,
  \end{cases}
\]
and the assertions follow.
\end{proof}

For later use, we record several direct consequences of the projected and
completion constructions.  First,
\[
  \dom T_{**}=\gph T_*.
\]
Moreover,
\begin{equation}\label{eq:projected-zero-identities}
  \zer T_*=P_*(\zer T),
  \qquad
  T_{**}(z_*,0)
  =\set{z_{**}\in H_{**}\mid (z_*,z_{**})\in\zer T}.
\end{equation}
Consequently,
\[
  \zer T_*\ne\emptyset
  \quad\Longleftrightarrow\quad
  \zer T\ne\emptyset.
\]
Define also the partial inverse $\widetilde T$ with respect to
$H_*\oplus H_{**}$ by
\begin{equation}\label{eq:partial-inverse}
  (w_*,z_{**})\in\widetilde T(z_*,w_{**})
  \quad\Longleftrightarrow\quad
  (w_*,w_{**})\in T(z_*,z_{**}).
\end{equation}
Then, it holds that 
\begin{equation}
\label{eq:partial-inverse-projection-fiber}
  T_*(z_*)=P_*\widetilde T(z_*,0),
  \qquad
  T_{**}(z_*,w_*)
  =
  \set{z_{**}\in H_{**}\mid
  (w_*,z_{**})\in\widetilde T(z_*,0)}.
\end{equation}
Thus $T_*(z_*)$ is the projection of the partial-inverse slice
$\widetilde T(z_*,0)$ onto $H_*$, whereas $T_{**}(z_*,w_*)$ is its
fiber over $w_*$. 

\begin{proposition}\label{prop:Tstar-basic}
If $T$ is maximal monotone, then $T_*$ and $T_{**}$ are convex-valued, and
$T_{**}$ has closed graph.
\end{proposition}

\begin{proof}
Since $T$ is maximal monotone, the partial inverse $\widetilde T$ is also maximal monotone
\cite[Proposition~2.1]{Spingarn1983}. Hence, it has closed convex values
\cite[Exercise~12.8(c)]{RockafellarWets1998}.
Relation~\eqref{eq:partial-inverse-projection-fiber} then gives convexity of
$T_*(z_*)$ and $T_{**}(z_*,w_*)$ \cite[Theorem 3.4]{Rockafellar1970ConvexAnalysis}.

Finally, suppose
$(z_*^\nu,w_*^\nu,z_{**}^\nu)\to(z_*,w_*,z_{**})$ and
$z_{**}^\nu\in T_{**}(z_*^\nu,w_*^\nu)$.  Then
$(w_*^\nu,0)\in T(z_*^\nu,z_{**}^\nu)$.  The graph of a maximal monotone
mapping is closed \cite[Exercise~12.8(b)]{RockafellarWets1998}, so $(w_*,0)\in T(z_*,z_{**})$, or equivalently
$z_{**}\in T_{**}(z_*,w_*)$.
\end{proof}

\section{Projected maximality and completion stability}\label{sec:conditions}

By Theorem~\ref{thm:implementability-equivalence}, implementability of the
semidefinite proximal step is equivalent to maximal monotonicity of $T_*$.  As
Example~\ref{ex:maximal-not-implementable} shows, this property cannot be
inferred from maximal monotonicity of $T$, even when $\zer T\ne\emptyset$.
In this section, we therefore develop verifiable conditions for the maximality
of $T_*$.  We begin with the classical relative-interior criterion and then
consider the piecewise-polyhedral case, where nonempty feasibility suffices.
Neither condition controls the kernel completions.  We therefore treat
boundedness of one completion fiber separately; that condition yields both
maximality of $T_*$ and local boundedness of $T_{**}$.

\subsection{Relative-interior and piecewise polyhedral conditions}

Let $J:H_*\hookrightarrow\HH$ denote the inclusion, so that $J^*=P_*$.  We
know directly from the definition of $T_*$,
\[
  z_*\in T_*^{-1}(w_*)
  \quad\Longleftrightarrow\quad
  \exists z_{**}\in H_{**}\ \text{such that}\
  (z_*,z_{**})\in T^{-1}(Jw_*).
\]
Consequently,
\[
  T_*^{-1}=J^*T^{-1}J.
\]
Since $\dom T^{-1}=\rge T$, the maximal-monotonicity preservation rule
\cite[Theorem~12.43]{RockafellarWets1998}, applied to the above identity, gives
the following classical baseline.

\begin{proposition}[Relative-interior sufficient condition]
\label{prop:ri-condition}
Let $T:\HH\To\HH$ be maximal monotone.  If
\begin{equation}\label{eq:ri-condition}
  H_*\cap\ri(\rge T)\ne\emptyset,
\end{equation}
then $T_*$ is maximal monotone. 
\end{proposition}

Condition \eqref{eq:ri-condition} is stronger than nonempty intersection:
$\zer T\ne\emptyset$ gives
$0\in H_*\cap\rge T$, but not necessarily
$0\in H_*\cap\ri(\rge T)$.  For piecewise-polyhedral mappings,
the relative-interior condition can be replaced by nonempty intersection
\cite{LiRockafellarSun2026Polyhedral}.  We quote the form needed here.

\begin{theorem}[{\cite[Theorem~3.3]{LiRockafellarSun2026Polyhedral}}]\label{thm:polyhedral-projection}
Let $T:\HH\To\HH$ be maximal monotone and piecewise polyhedral.  Suppose that 
\begin{equation}\label{eq:polyhedral-slice}
  H_*\cap\rge T\ne\emptyset.
\end{equation}
Then $T_*:H_*\To H_*$ is maximal monotone and piecewise polyhedral.
\end{theorem}

Since $\zer T\ne\emptyset$ implies \eqref{eq:polyhedral-slice}, the above theorem, together with Theorem~\ref{thm:implementability-equivalence}, gives the following immediate consequence.
\begin{corollary}\label{cor:polyhedral-implementability}
Let $T$ be maximal monotone and piecewise polyhedral.  If $\zer T\ne\emptyset$, then for every fixed splitting $\HH=H_*\oplus H_{**}$, $T_*$ is maximal monotone, and hence the aligned semidefinite proximal scheme for $T$ is implementable relative to the splitting.
\end{corollary}

The above corollary shows that, for piecewise polyhedral models, the implementability of the semidefinite proximal step is automatic whenever a solution exists, i.e., $\zer T\ne\emptyset$.

\subsection{One bounded fiber and completion stability}

Boundedness of a completion fiber is not required for implementability,
which depends only on maximal monotonicity of $T_*$.  It is instead a
sufficient condition that controls the otherwise unregularized completion
directions and, at the same time, provides a useful sufficient condition for
projected maximality.

\begin{assumption}[One bounded fiber]\label{ass:one-bounded-fiber}
There exists $(\bar z_*,\bar w_*)\in\gph T_*$ such that
$T_{**}(\bar z_*,\bar w_*)$ is bounded.
\end{assumption}

Note that this assumption holds automatically if $\zer T$ is nonempty and
bounded, by \eqref{eq:projected-zero-identities}.  We first isolate the
common recession geometry of the completion fibers.

\begin{lemma}
\label{lem:common-completion-recession}
Let $T:\HH\To\HH$ be maximal monotone and suppose 
$\gph T_* \ne\emptyset$. 
Define 
\begin{equation}\label{eq:Kstarstar}
	K_{**}:= \set{q_{**}\in H_{**}\mid \inner{(0,q_{**})}{y}\le0\quad\forall\,y\in D}, 
\end{equation}
where \(
D:=\cl(\dom\widetilde T)\) is nonempty closed convex with $\widetilde T$ being the partial inverse given in
\eqref{eq:partial-inverse}.
Then $K_{**}$ is a nonempty closed convex cone,
 and for all $(z_*,w_*)\in\gph T_*$, $T_{**}(z_*,w_*)$ is nonempty, closed and convex, and
  \begin{equation}
	\label{eq:common-completion-recession} \rec T_{**}(z_*,w_*) = \set{q_{**}\in H_{**}\mid
		(0,q_{**})\in N_D(z_*,0)} = K_{**}. 
\end{equation}
Consequently, one completion fiber is bounded if and only if every
completion fiber is bounded.  In that case every completion fiber is compact.
\end{lemma}

\begin{proof}
For any $(z_*, w_*) \in \gph T_* \ne \emptyset$, we know that $T_{**}(z_*, w_*)$ and $\widetilde T(z_*,0)$ are nonempty by \eqref{eq:Tstarstar-def} and \eqref{eq:partial-inverse}. By Proposition \ref{prop:Tstar-basic} and its proof, we know that the partial inverse
$\widetilde T$ is maximal monotone, and $T_{**}(z_*, w_*)$ and $\widetilde T(z_*,0)$ are
closed and convex. Moreover, $D$ is nonempty and convex since $\dom \widetilde{T}$ is nonempty and nearly convex \cite[Theorem~12.41]{RockafellarWets1998}. 
Hence, $K_{**}$ is a nonempty and closed convex cone by its definition.

Next, choose
$z_{**}\in T_{**}(z_*,w_*)$.  
For $q_{**} \in H_{**}$, the fiber identity \eqref{eq:partial-inverse-projection-fiber} implies that for all $t\ge 0$,
\[
z_{**} + t q_{**} \in T_{**}(z_*,w_*) \quad \Longleftrightarrow\quad (w_*,z_{**})+t(0,q_{**})\in\widetilde T(z_*,0).
\]
Hence, \( \rec T_{**}(z_*,w_*) =  \set{q_{**}\in H_{**}\mid
	(0,q_{**})\in \rec \widetilde T(z_*, 0)}.\)
Since $
N_D(z_*,0)
=\rec\widetilde T(z_*,0)
$ by \cite[Theorems~12.37 and 3.6]{RockafellarWets1998}, we have that 
\begin{equation}\label{eq:recTstar}
	\rec T_{**}(z_*,w_*) =  \set{q_{**}\in H_{**}\mid
	(0,q_{**})\in N_D(z_*,0)}.
\end{equation}
Note that for any $(z_*, w_*) \in \gph T_*$ and $q_{**} \in H_{**}$, $\inner{z_*}{q_{**}} = 0$, and
\[ \begin{aligned} (0, q_{**}) \in N_D(z_*,0) &\iff \inner{(0, q_{**})}{y-(z_*,0)}\le0 \qquad\forall\,y\in D\\ &\iff \inner{(0, q_{**})}{y}\le0 \qquad\forall\,y\in D\\ &\iff q_{**}\in K_{**}. \end{aligned} \]
This, together with \eqref{eq:recTstar}, proves \eqref{eq:common-completion-recession}.

Finally,
\cite[Theorem~8.4]{Rockafellar1970ConvexAnalysis} and
Proposition~\ref{prop:Tstar-basic} show that a completion fiber is bounded
exactly when its recession cone is $\{0\}$, and that every bounded completion
fiber is compact.
\end{proof}

The common-cone lemma now propagates Assumption~\ref{ass:one-bounded-fiber}
to the entire completion mapping.

\begin{theorem}\label{thm:one-bounded-fiber}
Let $T:\HH\To\HH$ be maximal monotone, and suppose
Assumption~\ref{ass:one-bounded-fiber} holds.  Then, 
\begin{enumerate}[label=(\roman*)]
  \item for every $(z_*,w_*)\in\gph T_*$, the completion fiber
  $T_{**}(z_*,w_*)$ is a nonempty compact convex subset of $H_{**}$;
  \item $T_{**}$ is locally bounded on its domain;
    \item $T_*$ is maximal monotone on $H_*$.
\end{enumerate}
\end{theorem}

\begin{proof}
	
Let $(\bar z_*,\bar w_*)$ be the graph point in Assumption~\ref{ass:one-bounded-fiber}, and retain $D$ and $K_{**}$ from \eqref{eq:Kstarstar}. Since $T_{**}(\bar z_*, \bar w_*)$ is bounded, Lemma~\ref{lem:common-completion-recession} implies that $K_{**} = \{0\}$, and for each $(z_*,w_*)\in\gph T_*$, $T_{**}(z_*,w_*)$ is nonempty closed convex and bounded.

To prove local boundedness of $T_{**}$ on $\dom T_{**}$, suppose to the contrary that for some $(\widehat z_*,\widehat w_*)\in\dom T_{**}$, there are 
\( (z_*^\nu,w_*^\nu)\to(\widehat z_*,\widehat w_*)\) in $\dom T_{**}$, and \( z_{**}^\nu\in T_{**}(z_*^\nu,w_*^\nu)\)
with $\norm{z_{**}^\nu}\to\infty$. 
Set
$\lambda_\nu:=1/\norm{z_{**}^\nu}$ and pass to a subsequence such that
$\lambda_\nu z_{**}^\nu\to q_{**}$ with $\norm{q_{**}}=1$. Then $\lambda_\nu(w_*^\nu,z_{**}^\nu)\to(0,q_{**})$.
 Since 
 $\lambda_\nu (w_*^\nu,z_{**}^\nu)\in \lambda_\nu \widetilde T(z_*^\nu,0)$ and  $(z_*^\nu,0)\to(\widehat z_*,0)$, the graphical convergence
 $\lambda\widetilde T\xrightarrow{g}N_D$ as $\lambda\downarrow0$
 \cite[Theorem~12.37]{RockafellarWets1998} yields
 $(0,q_{**})\in N_D(\widehat z_*,0)$. By \eqref{eq:common-completion-recession}, $q_{**}\in K_{**}=\{0\}$, contradicting $\norm{q_{**}}=1$. Hence $T_{**}$ is locally bounded on its domain.

Lastly,  since
$(\bar z_*,0)\in H_*\cap\dom\widetilde T$,  $H_*\cap D\ne\emptyset$.  We
claim that
\begin{equation}\label{eq:partial-inverse-ri-cq}
  H_*\cap\ri D\ne\emptyset.
\end{equation}
Otherwise, by \cite[Theorem~2.39 and Exercise~2.45(e)]{RockafellarWets1998}, proper separation of the convex set $D$ from the subspace $H_*$ would give a nonzero $(0,q_{**})\in H_{**} = H_*^\perp$ such that 
\[ \inner{(0,q_{**})}{y-(\bar z_*,0)}\le0 \qquad\forall\,y\in D. \] 
Since $\inner{(0,q_{**})}{(\bar z_*,0)} = 0$, this is equivalent to 
\[ \inner{(0,q_{**})}{y}\le0 \qquad\forall\,y\in D, \] 
i.e., $q_{**}\in K_{**}=\{0\}$, contradicting $(0,q_{**})\ne0$. Thus \eqref{eq:partial-inverse-ri-cq} holds. Let $J:H_*\hookrightarrow\HH$ be the canonical inclusion. Then, from \eqref{eq:partial-inverse-projection-fiber}, \( T_*=J^*\widetilde TJ\). Since $\ri(\dom\widetilde T)=\ri D$ by \cite[Theorem~12.41]{RockafellarWets1998},  \eqref{eq:partial-inverse-ri-cq} implies that
\[
H_* \cap \ri(\dom \widetilde T) \ne \emptyset.
\]
Now, \cite[Theorem~12.43]{RockafellarWets1998} asserts that $T_*$ is maximal monotone on $H_*$.
\end{proof}

\subsection{Relation to Ha's range condition and full-space admissibility}

The preceding results distinguish full-domain solvability of the active step
from boundedness or uniqueness of its kernel completions.  We now use this
distinction to compare three fixed-preconditioner conditions: Ha's range
condition \cite{Ha1990}, the admissibility condition of Bredies et al. \cite{BrediesChencheneLorenzNaldi2022}, and
Assumption~\ref{ass:one-bounded-fiber}.  Because the first two are stated for
a given semidefinite preconditioner $M$, rather than for a prescribed
splitting, we first identify the active and kernel spaces induced by $M$ and
express the full resolvent in terms of $T_*$ and $T_{**}$.  This yields an
exact characterization of admissibility and a hierarchy among the three
conditions.

For each self-adjoint positive semidefinite operator $M:\HH\to\HH$
considered in this subsection, the symbols $H_*,H_{**}$ and the associated
mappings $T_*,T_{**}$ are understood relative to the splitting induced by
that $M$:
\begin{equation}\label{eq:M-induced-splitting}
  H_*:=\rge M=(\ker M)^\perp,
  \qquad H_{**}:=\ker M,
  \qquad \HH=H_*\oplus H_{**}.
\end{equation}
The restriction $M_*:=M|_{H_*}$ is positive definite, and, under the
canonical product representation of \eqref{eq:M-induced-splitting},
\(
  M(z_*,z_{**})=(M_*z_*,0).
\)
Thus $M$ is aligned with the splitting it induces. 

In \cite{Ha1990}, the ambient space is already presented as the Cartesian
product $\HH=H_1\times H_2$, where $H_1$ and $H_2$ are finite-dimensional
Hilbert spaces, and 
\(
  M(x,y):=(0,y) \) for $x\in H_1,\ y\in H_2$.
Thus
\[
  H_*=\rge M=\{0\}\times H_2,
  \qquad
  H_{**}=\ker M=H_1\times\{0\}.
\]
Here $(x,y)$ is written in Ha's native $H_1\times H_2$ order; under the
active--completion representation of this splitting, the same vector is
represented by $((0,y),(x,0))\in H_*\times H_{**}$.
Ha assumes in \cite{Ha1990} that
\begin{equation}\label{eq:Ha-range-condition}
  0\in\operatorname{int}(\rge T).
\end{equation}
Under this condition, it is shown in \cite[Propositions~1--2]{Ha1990} that $(M+T)^{-1}M$ has full domain and that its active
component is single-valued and nonexpansive. Note that \eqref{eq:Ha-range-condition} implies that 
$0\in\operatorname{int}(\dom T^{-1})$, so $T^{-1}$ is locally bounded at
$0$ \cite{Rockafellar1969LocalBoundedness}. Hence
$\zer T=T^{-1}(0)$ is nonempty and bounded, and Assumption~\ref{ass:one-bounded-fiber} holds. 
On the other hand, Bredies et al. \cite{BrediesChencheneLorenzNaldi2022} call $M$ admissible for an operator $T$
when $(M+T)^{-1}M$ is both everywhere defined and single-valued
\cite[Definition~2.1]{BrediesChencheneLorenzNaldi2022}, which, together with Theorems \ref{thm:step-reduction} and \ref{thm:implementability-equivalence}, means that
\begin{equation}\label{eq:admissibility-completion-characterization}
  \begin{split}
  M\text{ is admissible for } T
  \quad\Longleftrightarrow\quad{}
  &T_*\text{ is maximal monotone, and}\\
  &T_{**}(z_*,w_*)\text{ is a singleton for every }
    (z_*,w_*)\in\gph T_*.
  \end{split}
\end{equation}
Indeed, for any $(z_*,w_*)\in \gph T_*$, set $x_* = z_* + M_*^{-1} w_* \in H_*$. Then, every $z_{**}\in T_{**}(z_*,w_*)$ gives $(z_*, z_{**})\in (M+T)^{-1}M(x_*,0)$. Hence single-valuedness of $(M+T)^{-1}M$ forces every nonempty completion fiber to be a singleton.
Thus, if $M$ is admissible for $T$, then Assumption~\ref{ass:one-bounded-fiber} holds as well.
We summarize these observations in the following proposition.

\begin{proposition}
\label{prop:admissibility-hierarchy}
Let $T:\HH\To\HH$ be maximal monotone, let $M$ be self-adjoint and
positive semidefinite.  If either $0\in\operatorname{int}(\rge T)$ or $M$
is admissible for $T$, then Assumption~\ref{ass:one-bounded-fiber} holds for
the splitting induced by $M$.
\end{proposition}

The next example shows that neither condition \eqref{eq:Ha-range-condition} nor condition \eqref{eq:admissibility-completion-characterization} is necessary for
Assumption~\ref{ass:one-bounded-fiber}. That is, these two conditions are strictly stronger than Assumption~\ref{ass:one-bounded-fiber}.

\begin{example}\label{ex:bounded-nonunique}
Let $\HH=\RR^2$, $H_*=\RR\times\{0\}$,
$H_{**}=\{0\}\times\RR$, and $M=\operatorname{diag}(1,0)$.  For
\[
  f(x,y):=\delta_{[-1,1]}(y),
  \qquad T:=\partial f,
\]
Assumption~\ref{ass:one-bounded-fiber} holds for the splitting induced by
$M$, although $M$ is not admissible for $T$ and condition
\eqref{eq:Ha-range-condition} fails.
\end{example}

\begin{proof}[{\bf Detail}]
Suppress the standard axis embeddings $s\mapsto(s,0)\in H_*$ and
$t\mapsto(0,t)\in H_{**}$ and write scalar coordinates for the arguments
and values of $T_*$ and $T_{**}$.  One has
\[
  T_*(x)=\{0\},
  \qquad
  T_{**}(x,0)=[-1,1].
\]
Thus Assumption~\ref{ass:one-bounded-fiber} holds.  For every
$(a,b)\in\RR^2$,
\[
  M(a,b)=(a,0),
  \qquad
  (M+T)^{-1}M(a,b)
  =\set{(a,y)\mid y\in[-1,1]}.
\]
Hence $(M+T)^{-1}M$ is not single-valued, so $M$ is not admissible for
$T$.  Moreover, $\rge T=\{0\}\times\RR$ has empty interior, so condition
\eqref{eq:Ha-range-condition} also fails.
\end{proof}

\section{Semidefinite PPA: convergence and applications}
\label{sec:convergence}
We now return to the semidefinite proximal iteration
\eqref{eq:intro-semidefinite-ppa} and study its convergence.  The proximal
operators $\{B_k\}$ are assumed to have a common range $H_*$ and kernel $H_{**}$.
The standing assumption below is stated for the original mapping $T$ and
separates solvability of the inclusion from implementability of the
semidefinite proximal steps.
\begin{assumption}\label{ass:Tstar-max}
The mapping $T:\HH\To\HH$ is maximal monotone, $\zer T\ne\emptyset$, and
the aligned semidefinite proximal scheme for $T$ is implementable relative
to the splitting $\HH=H_*\oplus H_{**}$.
\end{assumption}

By Theorem~\ref{thm:step-reduction}, each iteration \eqref{eq:intro-semidefinite-ppa} decomposes into a
variable-metric proximal step for $T_*$ on $H_*$ and a compatible completion
through $T_{**}$. We establish active convergence under a summable
residual-distance criterion and give conditions for full-space boundedness
and solution recovery.  We then identify the projected and completion
mappings for convex inf-projection and exact augmented Lagrangian steps.

\subsection{Active convergence and completion recovery}
\label{subsec:active-convergence}

The following standard assumption on the variable-metric sequence $\{B_k\}$ is needed for our analysis.
\begin{assumption}\label{ass:metric}
For each $k$, $c_k>0$ and $B_k:\HH\to\HH$ is self-adjoint and
positive semidefinite with common range $H_*$ and common kernel
$H_{**}$.  Denote the self-adjoint positive definite restriction of $B_k$
to $H_*$ by $B_{k,*}$.  There exist
$\alpha_k\ge1$ such that, for $k\ge1$,
\[
  \alpha_k^{-1}\norm{u}_{B_{k-1,*}}
  \le \norm{u}_{B_{k,*}}
  \le \alpha_k\norm{u}_{B_{k-1,*}}
  \qquad \forall\, u\in H_*,
\]
with $\prod_{k=1}^{\infty}\alpha_k<\infty$, and
\[
  c_k\longrightarrow c_\infty\in(0,+\infty].
\]
\end{assumption}

Under Assumption~\ref{ass:metric}, the sequence $\{B_{k,*}\}$ converges in
operator norm to a self-adjoint positive definite operator on $H_*$, denoted
by $B_{\infty,*}$; see \cite[(2.6)]{Rockafellar2023VariableMetric}. For later use, set
\[
  J_{k,*}:=\left(I+c_kB_{k,*}^{-1}T_*\right)^{-1}.
\]
Based on Theorem~\ref{thm:implementability-equivalence},
Assumption~\ref{ass:Tstar-max} ensures that $J_{k,*}$ is everywhere defined and
single-valued.  To allow inexact evaluation of $J_{k,*}$ at some given point, we define, for a computed iterate
$z^{k+1}=(z_*^{k+1},z_{**}^{k+1})$,
the active slice of the residual:
\[
  \mathcal R_{k,*}(z^{k+1})
  :=\set{r_*\in H_*\mid
  (r_*,0)\in
  T(z^{k+1})+c_k^{-1}B_k(z^{k+1}-z^k)},
\]
and set
\begin{equation}\label{eq:active-residual-distance}
  d_k:=\dist_{B_{k,*}^{-1}}
  \left(0,\mathcal R_{k,*}(z^{k+1})\right).
\end{equation}
Under Assumption \ref{ass:Tstar-max}, since $T$ is maximal monotone, each value $T(z^{k+1})$ is closed
and convex.  Hence $\mathcal R_{k,*}(z^{k+1})$ is closed and convex.
Whenever it is nonempty, the distance in
\eqref{eq:active-residual-distance} is attained.
Fix tolerances $\eps_k \in [0,1)$ with
$\sum_{k=0}^{\infty}\eps_k<\infty$.  The inexact semidefinite PPA accepts
$z^{k+1}$ when
\begin{equation}\label{eq:absolute-active-distance}
  c_kd_k\le \eps_k.
\end{equation}
An exact step is characterized by \(d_k=0\) and is enforced by taking \(\varepsilon_k=0\).

\begin{theorem}
\label{thm:semidefinite-convergence}
Let Assumptions~\ref{ass:Tstar-max} and \ref{ass:metric} hold, and let
$\{z^k\}\subset\HH$ satisfy \eqref{eq:absolute-active-distance}.  Then
\[
  z_*^k\to\bar z_*\in\zer T_*,
  \qquad
  z_*^{k+1}-z_*^k\to0.
\]

If, in addition, Assumption~\ref{ass:one-bounded-fiber} holds, then
$\{z^k\}$ is bounded and every cluster point belongs to $\zer T$.  If,
furthermore, $T_{**}(\bar z_*,0)=\{\bar z_{**}\}$, then
\[
  z_{**}^k\to\bar z_{**},
  \qquad
  z^k\to(\bar z_*,\bar z_{**}),
  \qquad
  0\in T(\bar z_*,\bar z_{**}).
\]
\end{theorem}

\begin{proof}
Set $q_*^{k+1}:=c_k^{-1}B_{k,*}(z_*^{k+1}-z_*^k)$.
If $r_*\in\mathcal R_{k,*}(z^{k+1})$, then
$(r_*-q_*^{k+1},0)\in T(z^{k+1})$, and hence
$r_*-q_*^{k+1}\in T_*(z_*^{k+1})$.  Consequently,
\[
  \mathcal R_{k,*}(z^{k+1})
  \subseteq
  T_*(z_*^{k+1})+q_*^{k+1},
\]
and \eqref{eq:absolute-active-distance} implies
\[
  c_k\dist_{B_{k,*}^{-1}}
  \left(0,T_*(z_*^{k+1})+q_*^{k+1}\right) \le c_k d_k
  \le \eps_k.
\]
This, together with the $B_{k,*}$-nonexpansivity of $J_{k,*}$ \cite{LiST2020LinearProgramming}, implies that
\[
  \norm{z_*^{k+1}-J_{k,*}(z_*^k)}_{B_{k,*}}
  \le
  c_k\dist_{B_{k,*}^{-1}}
  \left(0,T_*(z_*^{k+1})+q_*^{k+1}\right)
  \le\eps_k.
\]
Then, \cite[Theorem~2.3]{LiST2020LinearProgramming} and
\cite[Theorem~2.1]{Rockafellar2023VariableMetric} ensure that 
$z_*^k\to\bar z_*\in\zer T_*$ and
$z_*^{k+1}-z_*^k\to0$.

Since \eqref{eq:absolute-active-distance} makes $d_k$ finite, for each $k$, there exists $r_*^{k+1}\in\mathcal R_{k,*}(z^{k+1})$ such that
\(
  \norm{r_*^{k+1}}_{B_{k,*}^{-1}}
  = d_k.
\)
Set
\(
  \widehat w_*^{k+1}:=r_*^{k+1}-q_*^{k+1}.
\)
Then
\begin{equation}\label{eq:hatwkp1}
  (\widehat w_*^{k+1},0)\in T(z^{k+1}),
  \qquad
  z_{**}^{k+1}\in
  T_{**}(z_*^{k+1},\widehat w_*^{k+1}).
\end{equation}
Moreover, with $c_{\min}:=\inf_kc_k>0$,
\begin{equation}\label{eq:hatwkp10}
  \norm{\widehat w_*^{k+1}}_{B_{k,*}^{-1}}
  \le
  d_k 
  +c_k^{-1}\norm{z_*^{k+1}-z_*^k}_{B_{k,*}}
  \longrightarrow0,
\end{equation}
where $d_k\le\eps_k/c_{\min}\to0$.  Hence
$\widehat w_*^{k+1}\to0$.

Now suppose that Assumption~\ref{ass:one-bounded-fiber} holds.  By
Theorem~\ref{thm:one-bounded-fiber}, $T_{**}$ is locally bounded at
$(\bar z_*,0)$.  Therefore, $\{z_{**}^{k+1}\}$ is eventually bounded.
Together with the convergence of $\{z_*^{k+1}\}$,
this makes $\{z^{k}\}$ bounded. The assertion on the cluster point follows from the closedness of $\gph T$, \eqref{eq:hatwkp1}, and \eqref{eq:hatwkp10}. The singleton assumption on the limiting fiber $T_{**}(\bar z_*,0)$ then ensures that the full sequence converges to $(\bar z_*,\bar z_{**})$.
\end{proof}

\begin{remark}
\label{rem:active-convergence-interpretation}
The convergence of the active sequence $z_*^k$ can take different forms.
For example, suppose that the metrics $B_k$ are supplied in the structured 
mode of
\[
  B_k=\mathcal U G_k\mathcal U^*,
  \qquad \rge\mathcal U=H_*,
  \qquad G_k\succ0,
\]
with a fixed linear mapping $\mathcal U:E\to\HH$ from a finite-dimensional
Hilbert space $E$.  Then
$\ker\mathcal U^*=H_{**}$ and
$\mathcal U^*z^k=\mathcal U^*z_*^k$.  Since
$\mathcal U^*|_{H_*}$ is injective,
\[
  z_*^k\longrightarrow\bar z_*
  \quad\Longleftrightarrow\quad
  \mathcal U^*z^k\longrightarrow\mathcal U^*\bar z_*.
\]
Thus factor coordinates can identify the convergent component without
forming the projection $P_*$. 
\end{remark}

We also have the following standard variable-metric
rate result.

\begin{corollary}\label{cor:active-linear-rate}
In the setting of Theorem~\ref{thm:semidefinite-convergence}, suppose that the
same iterates satisfy, for some $\eps_k\in(0,1)$ with
$\sum_k\eps_k<\infty$, the tightened stopping rule
\begin{equation}\label{eq:relative-active-distance}
  c_kd_k
  \le \eps_k\min\set{1,
  \norm{z_*^{k+1}-z_*^k}_{B_{k,*}}}.
\end{equation}
Assume that $T_*$ is metrically subregular at the active limit $\bar z_*$
for $0$ in the $B_{\infty,*}$-metric.
Then $\dist_{B_{\infty,*}}(z_*^k,\zer T_*)$ converges Q-linearly at the rate
\begin{equation}\label{eq:generic-active-linear-rate}
  r=\frac{a_\infty}{\sqrt{a_\infty^2+c_\infty^2}},
\end{equation}
where \(
  a_\infty:=\operatorname{subreg}_{B_{\infty,*}}
  (T_*;\bar z_*\mid0)<\infty\) and $r=0$ if $c_\infty=+\infty$.
\end{corollary}

\begin{proof}
As in the proof of Theorem~\ref{thm:semidefinite-convergence},
\[
  \dist_{B_{k,*}^{-1}}
  \left(0,T_*(z_*^{k+1})+q_*^{k+1}\right)
  \le d_k.
\]
Thus \eqref{eq:relative-active-distance} is the relative stopping criterion
in \cite[Theorem~2.2]{Rockafellar2023VariableMetric} for $T_*$ on $H_*$. 
To match, if necessary, the normalization requirement \(c_k\ge 1\) in
\cite{Rockafellar2023VariableMetric}, set
\(
c_{\min}:=\inf_k c_k>0\),
\(T_*':=c_{\min}T_*\), and 
\(c_k':={c_k}/{c_{\min}}.
\)
Then $c_k'\ge 1$, and this scaling leaves the resolvents, the relative
stopping rule, and the rate quotient unchanged.
\end{proof}

\begin{remark}
\label{rem:inexact-residual-distance}
The stopping criteria  \eqref{eq:absolute-active-distance} and
\eqref{eq:relative-active-distance} do not require the algorithm to select a
particular member of \(\mathcal R_{k,*}(z^{k+1})\).  If an implementation or a subroutine
does produce an explicit residual
\(
  e_*^{k+1}\in\mathcal R_{k,*}(z^{k+1}),
\)
then
\[
  d_k\le\norm{e_*^{k+1}}_{B_{k,*}^{-1}}.
\]
Consequently, the computable bounds
\[
  c_k\norm{e_*^{k+1}}_{B_{k,*}^{-1}}\le\eps_k, \quad c_k\norm{e_*^{k+1}}_{B_{k,*}^{-1}}\le\eps_k \min\{1,\norm{z_*^{k+1}-z_*^k}_{B_{k,*}}\},
\]
are sufficient for \eqref{eq:absolute-active-distance} and \eqref{eq:relative-active-distance}, respectively.  
\end{remark}

\subsection{Convex minimization and augmented Lagrangian steps}
\label{subsec:convex-specializations}

We next identify the projected and completion mappings in two standard convex
settings.  For convex minimization, partial minimization produces the active
objective and recovers the omitted variables.  For the Lagrangian saddle
mapping, the same construction separates the multiplier update from its
primal realizations. 
In each case, the semidefinite metric determines which variables carry the
proximal dynamics and which are recovered through the completion.

\subsubsection{Convex minimization by inf-projection}

Let $f:\HH\to(-\infty,+\infty]$ be a proper closed convex function.  For the choice of $T = \partial f$, an exact semidefinite proximal step
\eqref{eq:intro-semidefinite-ppa} is equivalent to 
\begin{equation}\label{eq:semidefinite-ppa-minimization}
  \min_{z_*\in H_*,\,z_{**}\in H_{**}}
  \left\{
    f(z_*,z_{**})
    +\frac{1}{2c_k}\norm{z_*-z_*^k}_{B_{k,*}}^2
  \right\}.
\end{equation}
The quadratic term regularizes only $z_*$.  Eliminating the unregularized
variable $z_{**}$ therefore leads to the inf-projection
\begin{equation}\label{eq:finf-def}
  f_{\mathrm{inf}}(z_*):=\inf_{z_{**}\in H_{**}} f(z_*,z_{**}).
\end{equation}
Minimizing first over $z_{**}$ reduces the active problem to a proximal
minimization of $f_{\mathrm{inf}}$, while recovery of a full-space iterate
requires attainment of the infimum in \eqref{eq:finf-def}.  The next result shows how $T_*$ and $T_{**}$ are identified.

\begin{theorem}\label{thm:minimization-case}
Suppose there exists $\hat z_*\in\dom f_{\mathrm{inf}}$ such that
\[
  \argmin_{z_{**}\in H_{**}} f(\hat z_*,z_{**})
\]
is nonempty and bounded.  Then $f_{\mathrm{inf}}$ is closed, proper, and
convex, and for every $z_*\in\dom f_{\mathrm{inf}}$, the set
\[
  \argmin_{z_{**}\in H_{**}} f(z_*,z_{**})
\]
is nonempty and bounded.  Moreover,
\begin{equation}\label{eq:Tstar-subdiff}
  T_* = \partial f_{\mathrm{inf}},
\end{equation}
and
\begin{equation}\label{eq:Tstarstar-argmin}
  T_{**}(z_*,w_*)=
  \argmin_{z_{**}\in H_{**}} f(z_*,z_{**}) \qquad \forall \, (z_*,w_*)\in\gph T_*.
\end{equation}
\end{theorem}

\begin{proof}
Set $g(z_{**}):=f(\hat z_*,z_{**})$ for $z_{**}\in H_{**}$.  Boundedness of its nonempty minimizer
set and \cite[Proposition~3.23 and Theorem~3.21]{RockafellarWets1998} give
\[
  f^\infty(0,d_{**})=g^\infty(d_{**})>0
  \qquad\forall\, d_{**}\ne0.
\]
It follows from \cite[Theorem~3.31]{RockafellarWets1998} that $f$ is
level-bounded in $z_{**}$ locally uniformly in $z_*$.  Hence
\cite[Corollary~3.32 and Theorem~1.17(a)]{RockafellarWets1998} show that
$f_{\mathrm{inf}}$ is closed, proper, and convex and that, for every
$z_*\in\dom f_{\mathrm{inf}}$, the function $f(z_*,\cdot)$ attains its
minimum on a nonempty compact set.

Fix $z_*\in\dom f_{\mathrm{inf}}$ and
$z_{**}\in\argmin_y f(z_*,y)$.  For $w_*\in H_*$, the subgradient
inequality gives the equivalence
\begin{align*}
  w_*\in\partial f_{\mathrm{inf}}(z_*)
  &\quad\Longleftrightarrow\quad
  f_{\mathrm{inf}}(x_*)\ge f_{\mathrm{inf}}(z_*)+\inner{w_*}{x_*-z_*}
  \quad \forall \, x_*\in H_*\\
  &\quad\Longleftrightarrow\quad
  f(x_*,x_{**})\ge f(z_*,z_{**})
  +\inner{w_*}{x_*-z_*}
  \quad \forall \, (x_*,x_{**})\in\HH\\
  &\quad\Longleftrightarrow\quad
  (w_*,0)\in\partial f(z_*,z_{**}).
\end{align*}
It gives
\eqref{eq:Tstar-subdiff}  on
$\dom f_{\mathrm{inf}}$ and \eqref{eq:Tstarstar-argmin}.  If $z_*\notin\dom f_{\mathrm{inf}}$ and
$w_*\in T_*(z_*)$, a realizing relation
$(w_*,0)\in\partial f(z_*,z_{**})$ would make $z_{**}$ a finite minimizer of
$f(z_*,\cdot)$, a contradiction.  Thus both sides of
\eqref{eq:Tstar-subdiff} are empty outside $\dom f_{\mathrm{inf}}$.
\end{proof}

Under the assumptions of Theorem~\ref{thm:minimization-case},
\eqref{eq:Tstar-subdiff} and
\cite[Corollary~31.5.2]{Rockafellar1970ConvexAnalysis} show that $T_*$ is
maximal monotone.  Hence every aligned step is implementable.  Since
$f_{\mathrm{inf}}$ is proper, closed, and convex,
$\partial f_{\mathrm{inf}}(z_*)$ is nonempty for every
$z_*\in\ri(\dom f_{\mathrm{inf}})$.  Choose such a point $z_*$ and any
$w_*\in\partial f_{\mathrm{inf}}(z_*)$.  Then
\eqref{eq:Tstarstar-argmin} and Theorem~\ref{thm:minimization-case} show that
$T_{**}(z_*,w_*)$ is nonempty and bounded, so
Assumption~\ref{ass:one-bounded-fiber} holds.  Therefore,
Theorem~\ref{thm:one-bounded-fiber} shows that every completion fiber is
compact and that $T_{**}$ is locally bounded.

Under Assumption~\ref{ass:metric}, let $\{z^k\}$ be generated by
\eqref{eq:semidefinite-ppa-minimization}.  If $\argmin f\ne\emptyset$, then
Theorem~\ref{thm:semidefinite-convergence} gives
$z_*^k\to\bar z_*\in\argmin f_{\mathrm{inf}}$, boundedness of the full
iterates, and optimality of every cluster point.  The full sequence converges
whenever the limiting partial-minimizer set
\(
  \argmin_{z_{**}\in H_{**}}f(\bar z_*,z_{**})
\)
is a singleton.

\subsubsection{Augmented Lagrangian steps and primal realizability}

It is classical that the augmented Lagrangian method (ALM) is a
proximal point method on the dual side \cite{Rockafellar1976ALM}.
The proximal method of multipliers also provides a complementary
primal--dual interpretation \cite{Rockafellar2024PMM}.  In particular,
an exact augmented Lagrangian step can be viewed as a semidefinite
proximal point step for the Lagrangian saddle mapping.
We make this interpretation precise and characterize the associated
implementability and primal realizability.

Let $X$ and $Y$ be finite-dimensional Hilbert spaces and let
$\varphi:X\times Y\to(-\infty,+\infty]$ be closed, proper, and convex.  Define the corresponding Lagrangian and the dual functions by
\[
  \ell(x,y):=\inf_{u\in Y}
  \{\varphi(x,u)-\inner{y}{u}\},
  \qquad
  \psi(y):=\varphi^*(0,y)=-\inf_{x\in X}\ell(x,y).
\]
Since $\varphi$ is proper, closed, and convex, its conjugate
$\varphi^*$ is proper, closed, and convex
\cite[Theorem~12.2]{Rockafellar1970ConvexAnalysis}. Hence, its restriction $\psi(y) = \varphi^*(0,y)$ is also closed and convex. Define the dual subgradient mapping $T_g:=\partial\psi$ and the Lagrangian
saddle mapping $T_\ell:X\times Y\To X\times Y$ by
\begin{equation}\label{eq:Tl-Tg-def}
  T_\ell(x,y)
  :=\set{(v,u)\in X\times Y\mid (v,y)\in\partial\varphi(x,u)}.
\end{equation}
The mapping $T_\ell$ is the partial inverse of $\partial\varphi$ relative to
$X\times\{0\}$, and is therefore maximal monotone
\cite{Spingarn1983}.  On $\HH=X\times Y$, take
\begin{equation}\label{eq:alm-splitting-metric}
  H_*=\{0\}\times Y,
  \qquad H_{**}=X\times\{0\},
  \qquad B(x,y)=(0,C^{-1}y),
\end{equation}
where $C:Y\to Y$ is a self-adjoint positive definite linear operator. Then, $\HH = H_*\oplus H_{**}$.  Using the canonical factor isometries
$Y\ni y\mapsto(0,y)\in H_*$ and
$X\ni x\mapsto(x,0)\in H_{**}$, the projected and completion mappings of
$T_\ell$ are written as
\begin{align*}
  (T_\ell)_*(y)
  =\set{u\in Y\mid \exists x\in X:
                    (0,y)\in\partial\varphi(x,u)},\qquad
  (T_\ell)_{**}(y,u)
  =\set{x\in X\mid
                    (0,y)\in\partial\varphi(x,u)}.
\end{align*}
Thus $u\in(T_\ell)_*(y)$ if and only if
$(T_\ell)_{**}(y,u)\ne\emptyset$, and the latter set consists of all primal
points that realize the pair $(y,u)$.
For $c>0$, define the augmented Lagrangian function
\[
  \ell_{c,C}(x,y):=\inf_{u\in Y}
  \left\{\varphi(x,u)-\inner{y}{u}
        +\frac{c}{2}\norm{u}_C^2\right\}.
\]
The next result shows that the exact ALM step is a semidefinite proximal step for $T_\ell$ and characterizes its implementability.

\begin{proposition}\label{prop:alm-semidefinite-ppa}
Fix $c>0$.  For any current point
$(\widehat x,\widehat y)\in X\times Y$, a pair
$(x^+,y^+)\in X\times Y$ solves the $(c,B)$-step for $T_\ell$ on
$X\times Y$,
\begin{equation}\label{eq:alm-semidefinite-inclusion}
  0\in T_\ell(x^+,y^+)
  +c^{-1}B\bigl((x^+,y^+)-(\widehat x,\widehat y)\bigr),
\end{equation}
if and only if there exists $u^+\in Y$ such that
\begin{equation}\label{eq:exact-alm-subproblem}
  (x^+,u^+) \in\argmin_{x\in X,\,u\in Y}
  \left\{\varphi(x,u)-\inner{\widehat y}{u}
        +\frac{c}{2}\norm{u}_C^2\right\}, \qquad 
  y^+ =\widehat y-cCu^+ .
\end{equation}
Equivalently, the joint minimization in
\eqref{eq:exact-alm-subproblem} can be carried out in two stages:
\[
  x^+\in\argmin_{x\in X}\ell_{c,C}(x,\widehat y),
  \qquad
  u^+\in\argmin_{u\in Y}
  \left\{\varphi(x^+,u)-\inner{\widehat y}{u}
        +\frac{c}{2}\norm{u}_C^2\right\}.
\]
Whenever the step exists, $u^+$ and $y^+$ are uniquely determined, whereas
the set of possible primal components is exactly
\(
  (T_\ell)_{**}(y^+,u^+).
\)

The semidefinite proximal scheme for $T_\ell$, relative to the splitting
in \eqref{eq:alm-splitting-metric}, is implementable if and only if
$(T_\ell)_*$ is maximal monotone, which is also equivalent to the existence of a minimizer in
\eqref{eq:exact-alm-subproblem} for every $\widehat y\in Y$.  If, in
addition, $\psi$ is proper, these conditions are further equivalent to
$(T_\ell)_*=T_g$.  In that case,
\[
  y^+=(I+cCT_g)^{-1} \widehat y.
\]
\end{proposition}

\begin{proof}
Set $u^+:=c^{-1}C^{-1}(\widehat y-y^+)$.  By the definition of $B$,
\eqref{eq:alm-semidefinite-inclusion} is equivalent to
\[
  (0,u^+)\in T_\ell(x^+,y^+),
  \qquad y^+=\widehat y-cCu^+.
\]
By \eqref{eq:Tl-Tg-def}, this says
$(0,\widehat y-cCu^+)\in\partial\varphi(x^+,u^+)$, which is precisely the optimality
condition for \eqref{eq:exact-alm-subproblem}.  Minimizing first in $u$ gives
the equivalent marginal formulation.  The objective in
\eqref{eq:exact-alm-subproblem} is strongly convex in $u$, so all its
minimizers have the same $u^+$ and hence the same $y^+$.  The definition of
$(T_\ell)_{**}$ gives the asserted set of possible $x^+$.

The pointwise equivalence above, together with the fact that $\widehat x$
does not enter the step and
Theorem~\ref{thm:implementability-equivalence}, proves the implementability
characterization.  Finally,
$(0,y)\in\partial\varphi(x,u)$ implies
$(x,u)\in\partial\varphi^*(0,y)$, and hence $u\in\partial\psi(y)$.  Thus
$(T_\ell)_*\subset T_g$.  When $\psi$ is proper, its closedness and convexity
noted above imply that $T_g=\partial\psi$ is maximal monotone
\cite[Theorem~12.17]{RockafellarWets1998}.  The inclusion above then shows
that $(T_\ell)_*$ is maximal monotone if and only if $(T_\ell)_*=T_g$.
Finally, the last assertion follows from 
$y^+=(I+cC(T_\ell)_*)^{-1}\widehat y$.
\end{proof}

\begin{remark}\label{rem:alm-primal-realizability}
Suppose that $\psi$ in Proposition~\ref{prop:alm-semidefinite-ppa} is proper.  Then, the mapping $T_g=\partial\psi$ governs the
dual PPA \cite{Rockafellar1976PPA}, whereas $(T_\ell)_*$ retains only the dual graph pairs that admit a
primal realization.  Thus $(T_\ell)_*\subset T_g$ can be strict. Indeed, properness of
$\psi$ guarantees a dual proximal update, but not attainment of the
corresponding exact ALM subproblem.
\end{remark}

\begin{proof}[{\bf Detail}]
For example, let $X=Y=\RR$ and define
\[
  \varphi(x,u):=
  \begin{cases}
    u\exp(x/u), & 0<u\le 1,\\
    0, & u=0\text{ and }x\le0,\\
    +\infty, & \text{otherwise}.
  \end{cases}
\]
This is the closed perspective of the exponential function, restricted to
$u\le1$, and is therefore proper, closed, and convex.  Since
$\inf_x u\exp(x/u)=0$ for every $u>0$,
\[
  \psi(y)=\varphi^*(0,y)
  =\sup_{0\le u\le1}uy
  =\max\{0,y\}.
\]
Hence
\[
  T_g(y)=
  \begin{cases}
    \{0\}, & y<0,\\
    [0,1], & y=0,\\
    \{1\}, & y>0.
  \end{cases}
\]
If $u>0$, the first component of every subgradient of $\varphi$ at
$(x,u)$ is $\exp(x/u)>0$.  If $u=0$ and $x\le0$, the subgradient inequality
gives
\[
  (0,y)\in\partial\varphi(x,0)
  \quad\Longleftrightarrow\quad y\le0.
\]
Indeed, sufficiency follows from $\varphi\ge0$, whereas necessity follows by
fixing $u'>0$ and letting $x'\to-\infty$.  Consequently,
\[
  (T_\ell)_*(y)=
  \begin{cases}
    \{0\}, & y\le0,\\
    \emptyset, & y>0,
  \end{cases}
  \qquad
  \zer T_\ell=(-\infty,0]\times(-\infty,0]\ne\emptyset.
\]
Thus $(T_\ell)_*\subsetneq T_g$.  The completion fibers in this example satisfy
$(T_\ell)_{**}(y,0)=(-\infty,0]$ for $y\le0$ and are therefore unbounded. Consequently, given $c>0$, the dual proximal step associated with $T_g$ is well defined for every $\widehat y$. However, primal realizability may fail. In particular, when $\widehat y>0$, the corresponding exact ALM subproblem has no primal minimizer, i.e., its
infimum is finite but is approached only along sequences with
$x\to-\infty$.
\end{proof}

\section{Semidefinite progressive decoupling}\label{sec:applications}

Progressive decoupling applies the proximal point method to Spingarn's
partial inverse of a linkage problem.  We first recall the classical
positive definite step, then use a pair of brackets to retain proximal
regularization only on a prescribed active subspace.  This gives a
semidefinite PPA for the same partial inverse.  We subsequently derive its
realization in the original variables and apply the preceding
implementability theory.

\subsection{From the classical step to a semidefinite PPA}
\label{subsec:progressive-brackets}

Let $\mathcal M:\HH\To\HH$ be maximal monotone and let $S\subset\HH$ be a
linear subspace.  The associated linkage problem is
\begin{equation}\label{eq:linkage-problem}
  \text{find }(\bar z,\bar w)\in S\times S^\perp
  \quad\text{such that}\quad
  \bar w\in\mathcal M(\bar z).
\end{equation}
Let $\mathcal A:=\mathcal M^\#$ be Spingarn's partial inverse with respect to
$S$ \cite{Spingarn1983}.  Thus $(x^\#,u^\#)\in\gph\mathcal A$ if and only if
some $(z,w)\in\gph\mathcal M$ satisfies
\begin{equation}\label{eq:partial-inverse-coordinate-swap}
  x^\#=P_Sz+P_{S^\perp}w,
  \qquad
  u^\#=P_Sw+P_{S^\perp}z.
\end{equation}
Note that the mapping $\mathcal A$ is maximal monotone \cite[Proposition~2.1]{Spingarn1983}, and
\(
  \zer\mathcal A
  =
  \set{z+w\mid z\in S,\ w\in S^\perp,\ w\in\mathcal M(z)}.
\)
Hence \eqref{eq:linkage-problem} is equivalent to finding a zero of
$\mathcal A$.

Let $G_k:S\to S$ and $D_k:S^\perp\to S^\perp$ be self-adjoint and positive
definite, and set
\[
  \widetilde B_k:=G_kP_S+D_k^{-1}P_{S^\perp}.
\]
The classical progressive decoupling method is the following positive
definite PPA for $\mathcal A$ with metric $\widetilde B_k$
\cite{Rockafellar2019Progressive,Rockafellar2019ProgressiveD,Rockafellar2024PMM}:
\[
  0\in
  \mathcal A(x^{\#,k+1})
  +c_k^{-1}\widetilde B_k(x^{\#,k+1}-x^{\#,k}).
\]
Suppose that $x^{\#,k}=z_S^k+w_\perp^k$, where $z_S^k\in S$ and
$w_\perp^k\in S^\perp$.  The partial-inverse relation
\eqref{eq:partial-inverse-coordinate-swap} shows that the preceding
inclusion is equivalent to finding
$(\widehat z^{k+1},\widehat w^{k+1})\in\gph\mathcal M$ satisfying
\begin{equation*}\label{eq:positive-progressive-graph-step}
\begin{aligned}
  P_S\widehat w^{k+1}
  +c_k^{-1}G_k(P_S\widehat z^{k+1}-z_S^k)&=0,\\
  P_{S^\perp}\widehat z^{k+1}
  +c_k^{-1}D_k^{-1}
  (P_{S^\perp}\widehat w^{k+1}-w_\perp^k)&=0,
\end{aligned}
\end{equation*}
and then setting
\(
  x^{\#,k+1}
  :=P_S\widehat z^{k+1}+P_{S^\perp}\widehat w^{k+1}.
\)

To obtain a semidefinite variant, suppose that subspaces $S_0$ and $S_1$
are given with
\begin{equation}\label{eq:bracketing-subspaces}
  S_0\subset S\subset S_1,
  \qquad
  S_1^\perp\subset S^\perp\subset S_0^\perp.
\end{equation}
We use them as brackets for $S$.  Proximal regularization is retained on
the $S_0$-component of $z$ and the $S_1^\perp$-component of $w$, and the
remaining components are left unregularized.  These brackets induce the following
active--completion splitting.
Set
\(
  K_0:=S\cap S_0^\perp\),
\(  K_1:=S^\perp\cap S_1.
\)
Then
\[
\begin{aligned}
  S=S_0\oplus K_0, \quad
  S^\perp=K_1\oplus S_1^\perp,\quad 
  H_*:=S_0\oplus S_1^\perp, \quad 
  H_{**}:=K_0\oplus K_1,\quad 
  \HH=H_*\oplus H_{**}.
\end{aligned}
\]
Under \eqref{eq:partial-inverse-coordinate-swap}, the $H_*$-components of
$x^\#$ are $P_{S_0}z$ and $P_{S_1^\perp}w$, whereas its
$H_{**}$-components are $P_{K_0}z$ and $P_{K_1}w$.
Let $A_k:S_0\to S_0$ and
$C_k:S_1^\perp\to S_1^\perp$ be self-adjoint and positive definite.  Define
\begin{equation*}\label{eq:progressive-semidefinite-metric}
  B_k:=A_kP_{S_0}+C_k^{-1}P_{S_1^\perp}.
\end{equation*}
Relative to the decomposition
$\HH = S_0\oplus K_0\oplus K_1\oplus S_1^\perp$,
\[
  B_k=\operatorname{diag}(A_k,0,0,C_k^{-1}),
  \qquad
  \rge B_k=H_*,
  \qquad
  \ker B_k=H_{**},
  \qquad
  B_{k,*}=A_k\oplus C_k^{-1}.
\]
Relative to the positive definite metric $\widetilde B_k$, $B_k$ retains proximal
regularization only on $S_0$ and $S_1^\perp$ and vanishes on
$K_0\oplus K_1=H_{**}$.  The brackets therefore define the following
semidefinite PPA for the partial inverse $\mathcal A=\mathcal M^\#$.

\begin{algorithm}[Semidefinite PPA for the partial inverse]
\label{alg:semidefinite-progressive}
Choose $x^{\#,0}\in\HH$.  Given $x^{\#,k}$, choose
$x^{\#,k+1}$ satisfying
\begin{equation}\label{eq:progressive-abstract-step}
  0\in
  \mathcal A(x^{\#,k+1})
  +c_k^{-1}B_k(x^{\#,k+1}-x^{\#,k}).
\end{equation}
\end{algorithm}
Its realization, derived next, is the semidefinite
progressive decoupling step associated with the brackets $S_0$ and $S_1$.

\subsection{Bracketed realization and consequences}
\label{subsec:progressive-implementability}

Algorithm~\ref{alg:semidefinite-progressive} is stated on the partial-inverse
space.  Translating its inclusion through
\eqref{eq:partial-inverse-coordinate-swap} gives the associated semidefinite
progressive decoupling equations in the original variables of $\mathcal M$.
For an iterate $x^{\#,k}$, write
\(
  z_0^k:=P_{S_0}x^{\#,k} \),
  \(
  w_1^k:=P_{S_1^\perp}x^{\#,k}\).
  Then
  \(x_*^k=z_0^k+w_1^k.\)
The following proposition makes this translation precise.

\begin{proposition}\label{prop:progressive-step-realizations}
A point $x^{\#,k+1}$ satisfies
\eqref{eq:progressive-abstract-step} if and only if there is a pair
$(\widehat z^{k+1},\widehat w^{k+1})\in\gph\mathcal M$ such that
\begin{equation}\label{eq:progressive-graph-step}
\begin{aligned}
  & P_{S_0}\widehat w^{k+1}
  +c_k^{-1}A_k
  (P_{S_0}\widehat z^{k+1}-z_0^k) =0, \qquad 
  P_{K_0}\widehat w^{k+1} =0,\\
  & P_{K_1}\widehat z^{k+1} =0,\qquad
  P_{S_1^\perp}\widehat z^{k+1}
  +c_k^{-1}C_k^{-1}
  (P_{S_1^\perp}\widehat w^{k+1}-w_1^k)=0,
\end{aligned}
\end{equation}
and
\(
  x^{\#,k+1}
  =
  P_S\widehat z^{k+1}
  +P_{S^\perp}\widehat w^{k+1}.
\)
For such a realization, set
\[
\begin{aligned}
  x_*^{k+1}
  :=
  P_{S_0}\widehat z^{k+1}
  +P_{S_1^\perp}\widehat w^{k+1},\quad 
  x_{**}^{k+1}
  :=
  P_{K_0}\widehat z^{k+1}
  +P_{K_1}\widehat w^{k+1}.
\end{aligned}
\]
Then
\begin{align}
  0
  \in
  \mathcal A_*(x_*^{k+1})
  +c_k^{-1}B_{k,*}(x_*^{k+1}-x_*^k),
\qquad  x_{**}^{k+1}
  \in
  \mathcal A_{**}(x_*^{k+1},-c_k^{-1}B_{k,*}(x_*^{k+1}-x_*^k)).
  \label{eq:progressive-completion-step}
\end{align}
The system \eqref{eq:progressive-graph-step} is also equivalent to finding
$\widehat z^{k+1}$ satisfying
\begin{equation}\label{eq:progressive-constrained-inclusion}
\begin{split}
  0
  \in
  \mathcal M(\widehat z^{k+1})
  +N_{K_1^\perp}(\widehat z^{k+1}) +c_k^{-1}A_k
  (P_{S_0}\widehat z^{k+1}-z_0^k)
  +c_kC_kP_{S_1^\perp}\widehat z^{k+1} - w_1^k.
\end{split}
\end{equation}
\end{proposition}

\begin{proof}
Relative to the decomposition
$\HH = S_0\oplus K_0\oplus K_1\oplus S_1^\perp$, the four components of
$x^{\#,k+1}$ are
\[
  P_{S_0}\widehat z^{k+1},\quad
  P_{K_0}\widehat z^{k+1},\quad
  P_{K_1}\widehat w^{k+1},\quad
  P_{S_1^\perp}\widehat w^{k+1},
\]
and those of a corresponding value of $\mathcal A$ are
\[
  P_{S_0}\widehat w^{k+1},\quad
  P_{K_0}\widehat w^{k+1},\quad
  P_{K_1}\widehat z^{k+1},\quad
  P_{S_1^\perp}\widehat z^{k+1}.
\]
Thus \eqref{eq:progressive-abstract-step} is equivalent to
\eqref{eq:progressive-graph-step}, and \eqref{eq:progressive-completion-step}
follows from the definitions of $\mathcal A_*$ and
$\mathcal A_{**}$.

Lastly, we show the equivalence of \eqref{eq:progressive-graph-step} and \eqref{eq:progressive-constrained-inclusion}.  
From \eqref{eq:progressive-graph-step}, one has
\[
P_{{K_1}^\perp}\widehat w^{k+1} = c_k^{-1}A_k(z_0^k - P_{S_0}\widehat z^{k+1}) - c_kC_kP_{S_1^\perp}\widehat z^{k+1} + w_1^k \in {K_1}^\perp, \quad \widehat z^{k+1} \in K_1^\perp.
\]
Since $\widehat w^{k+1} \in \mathcal M(\widehat z^{k+1})$, one has
\[
0\in \mathcal M(\widehat z^{k+1}) - P_{K_1} \widehat w^{k+1} + c_k^{-1}A_k(P_{S_0}\widehat z^{k+1}-z_0^k) + c_kC_kP_{S_1^\perp}\widehat z^{k+1} - w_1^k,
\]
which implies \eqref{eq:progressive-constrained-inclusion} by noting that $-P_{K_1}\widehat w^{k+1} \in N_{K_1^\perp}(\widehat z^{k+1})= K_1$.  Conversely, if \eqref{eq:progressive-constrained-inclusion} holds, then $\widehat z^{k+1} \in K_1^\perp$ and there exist $\widehat w^{k+1} \in \mathcal M(\widehat z^{k+1})$ and $n^{k+1} \in N_{K_1^\perp}(\widehat z^{k+1}) $ such that
\[
0 = \widehat w^{k+1} + n^{k+1} + c_k^{-1}A_k(P_{S_0}\widehat z^{k+1}-z_0^k) + c_kC_kP_{S_1^\perp}\widehat z^{k+1} - w_1^k,
\]
which implies \eqref{eq:progressive-graph-step} by applying the projections $P_{S_0}$, $P_{K_0}$, and $P_{S_1^\perp}$ to both sides of the equation, respectively.
\end{proof}

The preceding representation identifies the bracketed equations with a
semidefinite proximal step for the partial inverse
$\mathcal A=\mathcal M^\#$.  We now apply the general criterion of
Section~\ref{sec:projection} to characterize its implementability.

\begin{corollary}
\label{thm:progressive-implementability}
Fix the brackets \eqref{eq:bracketing-subspaces}, $c>0$, and
positive definite $A:S_0\to S_0$ and $C:S_1^\perp\to S_1^\perp$.
The corresponding semidefinite progressive decoupling step is implementable
if and only if
\(
  \mathcal A_*=(\mathcal M^\#)_*\)
is maximal monotone on $H_*$.

\end{corollary}

Two conditions from Section~\ref{sec:conditions} are particularly useful.
If one completion fiber of $\mathcal A_{**}$ is bounded, then
Theorem~\ref{thm:one-bounded-fiber} gives implementability, makes every
completion fiber compact, and makes $\mathcal A_{**}$ locally bounded.  If
$\mathcal M$ is piecewise
polyhedral and \eqref{eq:linkage-problem} has a solution, then
$\mathcal A$ is piecewise polyhedral and
Corollary~\ref{cor:polyhedral-implementability} gives implementability, but
does not by itself control the completions. The second condition applies directly to block-separable piecewise linear-quadratic (PLQ) models.

\begin{corollary}\label{cor:block-plq-implementability}
Let $X_1,\ldots,X_m$ and $Y$ be finite-dimensional Hilbert spaces and set
$\HH:=X_1\times\cdots\times X_m$.  Let $L_i:X_i\to Y$ be linear, and consider
\[
  \min_{x=(x_1,\ldots,x_m)}
  \Phi(x):=\sum_{i=1}^m\phi_i(x_i)
  \quad\text{subject to}\quad
  Lx:=\sum_{i=1}^mL_ix_i=0,
\]
where every $\phi_i$ is closed, proper, convex, and piecewise
linear--quadratic.  Set $S:=\ker L$ and
$\mathcal M:=\partial\Phi$.  Suppose that there are $\tilde x$ and 
$\tilde y$ such that
\[
  L\tilde x=0,
  \qquad
  -L_i^*\tilde y\in\partial\phi_i(\tilde x_i)
  \quad(i=1,\ldots,m).
\]
Then, for every pair of brackets $S_0\subset S\subset S_1$, the mapping
$\mathcal A_*=(\mathcal M^\#)_*$ is maximal monotone, and the
semidefinite progressive decoupling step is implementable.
\end{corollary}

\begin{proof}
The mapping $\mathcal M$ is maximal monotone and piecewise polyhedral \cite[Proposition 12.30]{RockafellarWets1998}, and
the displayed KKT pair gives a solution of
\eqref{eq:linkage-problem}.  
Then \cite[Lemma 2.1]{LiRockafellarSun2026Polyhedral} implies that the partial inverse \(\mathcal A=\mathcal M^\#\) is also piecewise polyhedral.
Therefore, Corollaries \ref{cor:polyhedral-implementability} and \ref{thm:progressive-implementability} apply.
\end{proof}

We next give an explicit realization of the bracketed progressive decoupling step for a
linearly constrained separable convex problem.  Using the auxiliary-variable
construction of \cite[Section~5]{Rockafellar2026AdvancesPDA}, we obtain
independent local augmented Lagrangian subproblems followed by an averaging
step.  

\begin{example}\label{ex:alm-decomposition-semidefinite}
For $j=1,\ldots, m$, let $f_j:X_j\to(-\infty,+\infty]$ be proper, closed, and convex.
Given linear maps $L_j:X_j\to Y$ between finite-dimensional Hilbert spaces, consider
\begin{equation}\label{eq:almpda-model}
  \min_{x_1,\ldots,x_m}\left\{  \sum_{j=1}^m f_j(x_j)
  \mid \sum_{j=1}^m L_jx_j=b \right\},
\end{equation} 
where, to remove trivial cases, each $L_j$ is assumed to be nonzero.
For a given $r>0$, set self-adjoint positive semidefinite operators
\begin{equation}\label{eq:almpda-local-metric}
	Q_j:=\tau_jI-rL_j^*L_j\succeq0 \quad \mbox{with} \quad \tau_j\ge r\norm{L_j}^2, \qquad j=1,\ldots,m.
\end{equation}
Choose $b_j\in Y$ with $\sum_{j=1}^m b_j = b$, and initialize $x_j^0\in X_j$, $y^0\in Y$, and $u_j^0\in Y$ with
$\sum_{j=1}^m u_j^0=0$. The resulting semidefinite progressive decoupling iteration is, for $j=1,\ldots,m$:
\begin{equation}\label{eq:almpda-local-update}
\begin{aligned}
  & x_j^{k+1}=\operatorname{prox}_{f_j/\tau_j}(x_j^k-\tau_j^{-1}L_j^*
  \bigl[y^k+r(L_jx_j^k-b_j+u_j^k)\bigr]),
  \qquad a_j^{k+1} =L_jx_j^{k+1}-b_j,\\
  &
  u_j^{k+1}=-a_j^{k+1}+\frac1m\sum_{\ell=1}^m a_\ell^{k+1},\qquad
  y^{k+1}=y^k+ \frac{r}{m}\sum_{\ell=1}^m a_\ell^{k+1},
\end{aligned}
\end{equation}
where $ \operatorname{prox}_{f_j/\tau_j}(v):=\arg\min_x\{f_j(x)+(\tau_j/2)\|x-v\|^2\}.$
\end{example}

\begin{proof}[{\bf Detail}]
Write
$X:=\prod_jX_j$, $U:=Y^m$, and $U_0:=\{u\in U \mid \textstyle\sum_ju_j=0\}$.
Then, $U_0^\perp=\{\iota (y) \mid y\in Y\}$ and $P_{U_0^\perp}u=\iota(m^{-1}\sum_ju_j)$ with $
  \iota(y):=(y,\ldots,y).$
Consider the following lifted optimization problem:
\begin{equation}\label{eq:almpda-lifting}
 \min\left\{ \Phi(x,u):=\sum_{j=1}^m
       \bigl[f_j(x_j)+\delta_{\{0\}}(L_jx_j-b_j+u_j)\bigr] \mid (x,u)\in S:=X\times U_0 \right\}.
\end{equation}
Note that \eqref{eq:almpda-lifting} is equivalent to \eqref{eq:almpda-model}. Since $\Phi$ is proper,
closed, and convex,  $\mathcal M:=\partial\Phi$ is maximal monotone.
More explicitly, for
$(x,u),(v,y)\in X\times U$, it holds that
\begin{equation*}\label{eq:almpda-monotone-mapping}
  (v,y)\in\mathcal M(x,u)
  \quad\Longleftrightarrow\quad
  \left\{
  \begin{aligned}
    & L_jx_j-b_j+u_j=0,\\
    &v_j \in\partial f_j(x_j)+L_j^*y_j,
  \end{aligned}
  \right.
  \quad j=1,\ldots,m.
\end{equation*}
The associated linkage problem is 
\[
\text{find }((\bar x, \bar u), (0,\iota(\bar y)) )\in S\times S^\perp
\quad\text{such that}\quad
(0,\iota(\bar y)) \in\mathcal M(\bar x, \bar u).
\]
 Since $S^\perp=\{0\}\times U_0^\perp$, 
 \eqref{eq:partial-inverse-coordinate-swap} implies that
 \[
 (v,y)\in\mathcal M(x,u)  \quad\Longleftrightarrow\quad 
 u^\#\in\mathcal A(x^\#) \mbox { with } x^\# =(x,P_{U_0}u+P_{U_0^\perp}y), \quad u^\# =(v,P_{U_0}y+P_{U_0^\perp}u).
 \]

We now choose the brackets and metric.  Let
$Q:=\operatorname{diag}(Q_1,\ldots,Q_m)$, and take
\[
  S_0:=\rge Q \times U_0,\qquad S_1:=S,\qquad c_k\equiv1,
\]
with $A_k=Q|_{\rge Q}\oplus rI_{U_0}$ and $C_k=rI_{U_0^\perp}$.
Thus, with respect to the product coordinates $(x,u)\in\HH=X\times U$, we obtain
\begin{equation*}\label{eq:almpda-partial-inverse-metric}
  B_k=\begin{bmatrix}
  	Q & 0 \\
  	0 & rP_{U_0} + r^{-1} P_{U_0^\perp}
  \end{bmatrix},\qquad
  H_*=\rge Q\oplus U,\qquad
  H_{**}=\ker Q\oplus\{0\}.
\end{equation*}
Then, the bracketed semidefinite progressive decoupling step \eqref{eq:progressive-graph-step} can be written as finding
\[
  (v^{k+1},\widehat y^{k+1})
  \in\mathcal M(x^{k+1},\widehat u^{k+1})
\]
such that
\begin{equation*}\label{eq:almpda-graph-step}
\begin{aligned}
   & v^{k+1}+Q(x^{k+1}-x^k) =0,\quad
  P_{U_0}\widehat y^{k+1}
    +r(P_{U_0}\widehat u^{k+1}-u^k) =0,\\
   &P_{U_0^\perp}\widehat u^{k+1}
    +r^{-1}(P_{U_0^\perp}\widehat y^{k+1}-\iota y^k) =0.
\end{aligned}
\end{equation*}
The last two equations combine to give
\[
  \widehat y^{k+1}=\iota (y^k)-r(\widehat u^{k+1}-u^k).
\]
Together with $\mathcal M=\partial\Phi$, \eqref{eq:progressive-constrained-inclusion} shows that
$(x^{k+1},\widehat u^{k+1})$ is obtained by solving, for each $j$, the independent
minimization problem
\begin{equation*}\label{eq:almpda-all-variable-subproblem}
  \min_{x_j,u_j}\left\{
    f_j(x_j)+\delta_{\{0\}}(L_jx_j-b_j+u_j)
    -\inner{y^k}{u_j}
    +\frac12\norm{x_j-x_j^k}_{Q_j}^2
    +\frac r2\norm{u_j-u_j^k}^2
  \right\},
\end{equation*}
which can be recast into the following augmented Lagrangian subproblem
\begin{equation*}\label{eq:almpda-local-subproblem}
  \min_{x_j}\left\{
    f_j(x_j)+\inner{y^k}{L_jx_j-b_j}
    +\frac r2\norm{L_jx_j-b_j+u_j^k}^2
    +\frac12\norm{x_j-x_j^k}_{Q_j}^2
  \right\}.
\end{equation*}
The choice of $Q_j$ in  \eqref{eq:almpda-local-metric} further implies that the above subproblem is strongly convex and its unique minimizer $x_j^{k+1}$ satisfies
\[
  0\in\partial f_j(x_j^{k+1})+\tau_j(x_j^{k+1}-x_j^k)
       +L_j^*\bigl[y^k+r(L_jx_j^k-b_j+u_j^k)\bigr].
\]
This results in precisely the proximal update of $x_j^{k+1}$ in \eqref{eq:almpda-local-update}.
Then, for $j=1,\ldots,m$,
\[
  \widehat u_j^{k+1}=-a_j^{k+1},\quad
  \widehat y_j^{k+1} =y^k+r(a_j^{k+1}+u_j^k) \mbox{ with } a_j^{k+1} = L_j x_j^{k+1} - b_j,
\]
and 
\[
  x^{\#,k+1}
  =\bigl(x^{k+1},P_{U_0}\widehat u^{k+1}
                      +P_{U_0^\perp}\widehat y^{k+1}\bigr)
  =\bigl(x^{k+1},u^{k+1}+\iota (y^{k+1})\bigr),
\]
where
\[
  u_j^{k+1}=-a_j^{k+1} +  \frac{1}{m}\sum_{\ell=1}^m a_\ell^{k+1} ,
  \qquad
  y^{k+1}=\frac1m\sum_j\widehat y_j^{k+1}
         =y^k+\frac{r}{m}\sum_{\ell=1}^m a_\ell^{k+1}.
\]
We thus recover the updates in \eqref{eq:almpda-local-update}.
\end{proof}

\section{Conclusion}\label{sec:conclusion}

In this paper, we have developed an active--completion view of semidefinite PPA. A semidefinite proximal term is selective rather than merely degenerate: the proximal dynamics take place only on the active space, while the remaining directions are handled through compatible completions, which need not be unique.
This perspective arises naturally in inf-projection, augmented Lagrangian, and progressive decoupling constructions. In this sense, semidefinite PPA provides a way to adapt proximal regularization to the structure of the problem, rather than simply a singular-metric version of the classical PPA.

\end{document}